\documentclass[reqno]{amsart}
\usepackage{hyperref} 

\begin{document}
\title[Positive solutions of resonant singular equations]
{Pairs of positive solutions for resonant singular equations with the $p$-Laplacian}

\author[N. S. Papageorgiou, V. D. R\u{a}dulescu, D. D. Repov\v{s}]
{Nikolaos S. Papageorgiou, Vicen\c{t}iu D. R\u{a}dulescu, Du\v{s}an D. Repov\v{s}}

\address{Nikolaos S. Papageorgiou \newline
National Technical University,
Department of Mathematics,
Zografou Campus, Athens 15780, Greece}
\email{npapg@math.ntua.gr}

\address{Vicen\c{t}iu D. R\u{a}dulescu (corresponding author) \newline
Department of Mathematics, Faculty of Sciences,
King Abdulaziz University, P.O. Box 80203,
Jeddah 21589, Saudi Arabia. \newline
 Department of Mathematics,
University of Craiova,
Craiova 200585, Romania}
\email{radulescu@inf.ucv.ro}

\address{Du\v{s}an D. Repov\v{s} \newline
Faculty of Education and Faculty of Mathematics and Physics,
University of Ljubljana, Ljubljana 1000, Slovenia}
\email{dusan.repovs@guest.arnes.si}

\subjclass[2010]{35J20, 35J25, 35J67}
\keywords{Singular reaction; resonance; regularity; positive solutions;
\hfill\break\indent maximum principle; mountain pass theorem}

\begin{abstract}
 We consider a nonlinear elliptic equation driven by the Dirichlet
 $p$-Laplacian with a singular term and a $(p-1)$-linear perturbation
 which is resonant at $+\infty$ with respect to the principal eigenvalue.
 Using variational tools, together with suitable truncation and comparison
 techniques, we show the existence of at least two positive smooth solutions.
\end{abstract}

\maketitle
\numberwithin{equation}{section}
\newtheorem{theorem}{Theorem}[section]
\newtheorem{proposition}[theorem]{Proposition}
\newtheorem{remark}[theorem]{Remark}
\newtheorem{example}[theorem]{Example}
\allowdisplaybreaks

\section{Introduction}

Let $\Omega\subseteq{\mathbb R}^N$ be a bounded domain with a $C^2$-boundary
$\partial\Omega$. In this paper, we study the following nonlinear elliptic
problem with singular reaction
\begin{equation}\label{eq1}
 \begin{gathered}
  -\Delta_pu(z)=u(z)^{-\mu}+f(z,u(z))\quad \text{in } \Omega,\\
  u|_{\partial\Omega}=0,\quad u>0,\quad 1<p<\infty,\; 0<\mu<1.
 \end{gathered}
\end{equation}
In this problem, $\Delta_p$ denotes the $p$-Laplacian differential operator
defined by
$$
\Delta_pu={\rm div}\,(|Du|^{p-2}Du)\quad \text{for all }
 u\in W^{1,p}(\Omega),\ 1<p<\infty.
$$
In the reaction term, $u^{-\mu}$ (with $0<\mu<1$) is the singular part and
$f:\Omega\times{\mathbb R}\to{\mathbb R}$ is a Carath\'eodory
perturbation (that is, for all $x\in{\mathbb R}$ the mapping $z\mapsto f(z,x)$
is measurable and for almost all $z\in\Omega$ the map $x\mapsto f(z,x)$
is continuous) which exhibits $(p-1)$-linear growth near $+\infty$.

Using variational tools, together with suitable truncation and comparison
techniques, we prove a multiplicity theorem establishing the existence of
two positive smooth solutions. Such multiplicity theorems for singular
problems were proved by Hirano, Saccon and Shioji \cite{7}, Papageorgiou
and R\u adulescu \cite{12}, Sun, Wu and Long \cite{16}
(semilinear problems driven by the Laplacian) and Giacomoni and
Saudi \cite{4}, Giacomoni, Schindler and Takac \cite{5}, Kyritsi and Papageorgoiu
 \cite{6}, Papageorgiou and Smyrlis \cite{13, 14}, Perera and Zhang \cite{15}
(nonlinear problems). In all these papers the reaction term is parametric.
The presence of the parameter permits a more precise control of the nonlinearity
as the positive parameter $\lambda$ becomes small.

A complete overview of the theory of singular elliptic equations can be
found in the book by Ghergu and R\u adulescu \cite{3}.

\section{Mathematical background and hypotheses}

Let $X$ be a Banach space and $X^*$ its topological dual.
 By $\langle \cdot,\cdot\rangle$ we denote the duality brackets for the
pair $(X^*,X)$. Given $\varphi\in C^1(X,{\mathbb R})$, we say that
 $\varphi$ satisfies the ``Cerami condition" (the ``C-condition" for short),
if the following property holds:
\begin{quote}
Every sequence $\{u_n\}_{n\geq 1}\subseteq X$ such that
$\{\varphi(u_n)\}_{n\in {\mathbb N}}\subseteq{\mathbb R}$ is bounded
and $(1+\|u_n\|)\varphi'(u_n)\to 0$ in $X^*$ as $n\to\infty$,
admits a strongly convergent subsequence.
\end{quote}

This is a compactness-type condition on the functional $\varphi$.
It leads to a deformation theorem from which we can deduce the minimax theory
of the critical values of $\varphi$. One of the main results of this theory
is the so-called ``mountain pass theorem'', which we recall here.

\begin{theorem}\label{thm1}
Assume that $\varphi\in C^1(X,{\mathbb R})$ satisfies the C-condition,
$0<\rho<\|u_0-u_1\|$,
 $$
\max\{\varphi(u_0),\varphi(u_1)\}< \inf\{\varphi(u):\|u-u_0\|=\rho\}=m_{\rho}
$$
 and
$$
c=\inf_{\gamma\in\Gamma}\max_{0\leq t\leq 1}\varphi(\gamma(t)),
$$
where
$$
\Gamma=\{\gamma\in C([0,1],X):\gamma(0)=u_0,\gamma(1)=u_1\}.
$$
Then $c\geq m_{\rho}$ and $c$ is a critical value of $\varphi$
(that is, there exists $u_0\in X$ such that $\varphi(u_0)=c$ and
$\varphi'(u_0)=0$).
\end{theorem}

In the analysis of problem \eqref{eq1} we will use the Sobolev space
$W^{1,p}_0(\Omega)$ and the Banach space
$C^1_0(\overline{\Omega})=\{u\in C^1(\overline{\Omega}):u|_{\partial\Omega}=0\}$.
In what follows, we denote by $\|\cdot\|$ the norm of the Sobolev space
$W^{1,p}_0(\Omega)$. On account of the Poincar\'e inequality, we have
$$
\|u\|=\|Du\|_p\quad \text{for all } u\in W^{1,p}_0(\Omega).
$$

The Banach space $C^1_0(\overline{\Omega})$ is an ordered Banach space
with positive (order) cone given by
$$
C_+(\overline{\Omega})=C_+=\{u\in C^1_0(\overline{\Omega}):u(z)
\geq 0\ \text{for all}\ z\in\Omega\}.
$$
This cone has a nonempty interior
$$
\operatorname{int}C_+=\big\{u\in C_+:u(z)>0 \text{ for all }
 z\in\Omega,\frac{\partial u}{\partial n}\big|_{\partial\Omega}<0\big\}.
$$
Here, $\frac{\partial u}{\partial n}=(Du,n)_{{\mathbb R}^N}$ with $n(\cdot)$
being the outward unit normal on $\partial\Omega$.

Let $A:W^{1,p}_0(\Omega)\to W^{-1,p'}(\Omega)=W^{1,p}_0(\Omega)^*$
(with $\frac{1}{p}+\frac{1}{p'}=1$) be the nonlinear map defined by
$$
\langle A(u),h\rangle=\int_{\Omega}|Du|^{p-2}(Du,Dh)_{{\mathbb R}^N}dz\quad
 \text{for all } u,h\in W^{1,p}_0(\Omega).
$$
This map has the following properties (see, for example, Motreanu,
Motreanu and Papageorgiou \cite[p. 40]{11}).

\begin{proposition}\label{prop2}
 The map $A:W^{1,p}_0(\Omega)\to W^{-1,p'}(\Omega)$ is bounded
(that is, maps bounded sets to bounded sets), continuous, strictly monotone
(hence maximal monotone, too) and of type $(S)_+$, that is,
\[
u_n\stackrel{w}{\to}u \text{ in } W^{1,p}_0(\Omega) \text{ and }
 \limsup_{n\to\infty}\langle A(u_n),u_n-u\rangle
\leq 0\Rightarrow u_n\to u \text{ in } W^{1,p}_0(\Omega).
\]
\end{proposition}

We will also need some facts about the spectrum of the Dirichlet
$p$-Laplacian. So, we consider the following nonlinear eigenvalue problem
$$
-\Delta_pu(z)=\hat{\lambda}m(z)|u(z)|^{p-2}u(z) \text{ in } \Omega,\quad
 u|_{\partial\Omega}=0.
$$

Here, $m\in L^{\infty}(\Omega),m\geq 0,m\neq 0$. We say that $\hat{\lambda}$
is an ``eigenvalue", if the above problem admits a nontrivial solution $\hat{u}$
known as an ``eigenfunction'' corresponding to the eigenvalue $\hat{\lambda}$.
 The nonlinear regularity theory (see, for example, Gasinski and
Papageorgiou \cite[pp. 737-738]{2}), implies that
$\hat{u}\in C^1_0(\overline{\Omega})$. There exists a smallest eigenvalue
$\hat{\lambda}_1(m)$ such that:
\begin{itemize}
 \item $\hat{\lambda}_1(m)>0$ and is isolated in the spectrum
$\hat{\sigma}(p)$ of ($-\Delta_p,W^{1,p}_0(\Omega),m$) (that is, there exists
 $\epsilon>0$ such that
 $(\hat{\lambda}_1(m),\hat{\lambda}_1(m)+\epsilon)\cap\hat{\sigma}(p)=\emptyset$);

 \item $\hat{\lambda}_1(m)>0$ is simple in the sense that if $\hat{u},\hat{v}$
 are two eigenfunctions corresponding to $\hat{\lambda}_1(m)>0$, then
$\hat{u}=\xi\hat{v}$ for some $\xi\in{\mathbb R}\backslash\{0\}$;

\item
 \begin{equation}\label{eq2}
   \hat{\lambda}_1(m)=\inf\Big[\frac{\|Du\|^p_p}{\int_{\Omega}m(z)|u|^pdz}:
u\in W^{1,p}_0(\Omega),u\neq 0\Big].
 \end{equation}
\end{itemize}

The infimum in \eqref{eq2} is realized on the one-dimensional eigenspace
corresponding to $\hat{\lambda}_1(m)$. From the above properties it follows
that the elements of this eigenspace have constant sign.
We denote by $\hat{u}_1(m)$ the $L^p$-normalized (that is, $\|\hat{u}_1(m)\|_p=1$)
positive eigenfunction for the eigenvalue $\hat{\lambda}_1(m)$.
As we have already mentioned, $\hat{u}_1(m)\in C_+$. In fact, the nonlinear
 maximum principle (see, for example, Gasinski and Papageorgiou \cite[p. 738]{2})
implies that $\hat{u}_1(m)\in \operatorname{int}C_+$. If $m\equiv 1$, then we write
$$
\hat{\lambda}_1(1)=\hat{\lambda}_1>0\quad \text{and}\quad
 \hat{u}_1(1)=\hat{u}_1\in \operatorname{int}C_+.
$$

The map $m\mapsto\hat{\lambda}_1(m)$ exhibits the following strict monotonicity
 property.

\begin{proposition}\label{prop3}
 If $m_1,m_2\in L^{\infty}(\Omega),\ 0\leq m_1(z)\leq m_2(z)$ for almost all
$z\in\Omega$ and $m_1\neq 0,m_2\neq m_1$, then
$\hat{\lambda}_1(m_2)<\hat{\lambda}_1(m_1)$.
\end{proposition}

We mention that every eigenfunction $\hat{u}$ corresponding to an eigenvalue
$\hat{\lambda}\neq\hat{\lambda}_1(m)$, is necessarily nodal
(that is, sign changing). For details on the spectrum of
$(-\Delta_p, W^{1,p}_0(\Omega),m)$ we refer to \cite{2, 11}.

For $x\in{\mathbb R}$ we define $x^{\pm}=\max\{\pm x,0\}$.
Then, given $u\in W^{1,p}_0(\Omega)$, we set $u^{\pm}(\cdot)=u(\cdot)^{\pm}$.
We have
$$
u^{\pm}\in W^{1,p}_0(\Omega),\ u=u^+-u^-,\ |u|=u^++u^-.
$$
Given a measurable function $g:\Omega\times{\mathbb R}\to{\mathbb R}$
(for example, a Carath\'eodory function), we denote by $N_g(\cdot)$
the Nemitsky (superposition) operator corresponding to $g$, that is,
 $$
N_g(u)(\cdot)=g(\cdot,u(\cdot))\ \text{for all}\ u\in W^{1,p}_0(\Omega).
$$
We know that $z\mapsto N_g(u)(z)=g(z,u(z))$ is measurable.

The hypotheses on the perturbation term $f(z,x)$ are the following:
\smallskip

\noindent(H1):\ $f:\Omega\times{\mathbb R}\to{\mathbb R}$ is a Carath\'eodory
function such that $f(z,0)=0$ for almost all $z\in\Omega$ and
\begin{itemize}
 \item[(i)] for every $\rho>0$, there exists $a_{\rho}\in L^{\infty}(\Omega)$
such that
 $$
|f(z,x)|\leq a_{\rho}(z)\quad \text{for almost all } z\in\Omega, \text{ all }
 0\leq x\leq\rho
$$
 and there exists $w\in C^1(\overline{\Omega})$ such that
 $$
w(z)\geq\hat{c}>0 \text{ for all } z\in\overline{\Omega} \text{ and }
 -\Delta_pw\geq 0 \text{ in } W^{1,p}_0(\Omega)^*=W^{-1,p'}(\Omega)
$$
 and for every compact $K\subseteq \Omega$ we can find $c_K>0$ such that
 $$
w(z)^{-\mu}+f(z,w(z))\leq-c_K<0 \text{ for almost all } z\in K;
$$

 \item[(ii)] if $F(z,x)=\int^x_0f(z,s)ds$, then there exists
$\eta\in L^{\infty}(\Omega)$ such that
 \begin{gather*}
\hat{\lambda}_1\leq\liminf_{x\to+\infty}\frac{f(z,x)}{x^{p-1}}
\leq\limsup_{x\to+\infty}\frac{f(z,x)}{x^{p-1}}\leq \eta(z)
\text{ uniformly for almost all } z\in\Omega,\\
f(z,x)x-pF(z,x)\to-\infty\text{ as } x\to+\infty \text{ uniformly for almost all }
 z\in\Omega;
 \end{gather*}

 \item[(iii)] there exists $\delta\in(0,\hat{c})$ such that for all compact
$K\subseteq\Omega$ we have
 $$
f(z,x)\geq\hat{c}_K>0\ \text{ for almost all } z\in K, \text{ all } 0<x\leq \delta;
$$

 \item[(iv)] for every $\rho>0$, there exists $\hat{\xi}_{\rho}>0$ such that
for almost all $z\in\Omega$ the mapping
 $$
x\mapsto f(z,x)+\hat{\xi}_{\rho}x^{p-1}
$$
 is nondecreasing on $[0,\rho]$.
\end{itemize}

\begin{remark} \rm
 Since we are looking for positive solutions and all the above hypotheses concern
the positive semiaxis ${\mathbb R}_+=\left[0,+\infty\right)$, we may assume
without any loss of generality that
 \begin{equation}\label{eq3}
  f(z,x)=0 \text{ for almost all } z\in\Omega,\text{ all } x\leq 0.
 \end{equation}
\end{remark}

Hypothesis (H1)(ii) permits resonance with respect to the principal eigenvalue
$\hat{\lambda}_1>0$. The second convergence condition in (H1)(ii)
 implies that the resonance at $+\infty$ with respect to $\hat{\lambda}_1>0$,
 is from the right of the principle eigenvalue in the sense that
$$
\hat{\lambda}_1x^{p-1}-pF(z,x)\to-\infty\ \text{as}\ x\to+\infty
\text{ uniformly for almost all } z\in\Omega
$$
(see the proof of Proposition \ref{prop5}). This makes the problem noncoercive
and so the direct method of the calculus of variations is not applicable.

Hypothesis (H1)(iv) is satisfied if for example $f(z,\cdot)$ is
differentiable and the derivative $f'_x(z,\cdot )$ satisfies for some $\rho>0$
$$
f'_x(z,x)\geq -\tilde{c}_\rho x^{p-2} \text{ for almost all $x\in\Omega$, for all
$0\leq x\leq\rho$ and some $\tilde{c}_\rho>0$.}
$$

\begin{example} \rm
The following function satisfies hypotheses (H1). For the sake of simplicity
we drop the $z$-dependence:
$$
f(x)=\begin{cases}
 x^{p-1}-2x^{r-1}&\text{if } 0\leq x\leq 1\\
 \eta x^{p-1}+x^{\tau-1}-(2+\eta)x^{q-1}&\text{if } 1<x,
\end{cases}
$$
with $\eta\geq\hat{\lambda}_1$ and $1<\tau,\ q<p<r<\infty$.
\end{example}

\section{Pair of positive solutions}

In this section we prove the existence of two positive smooth solutions
for problem \eqref{eq1}.
We start by considering the  auxiliary singular Dirichlet problem
\begin{equation}\label{eq4}
 -\Delta_pu(z)=u(z)^{-\mu} \text{ in } \Omega,\ u|_{\partial\Omega}=0,\ u>0.
\end{equation}
By Papageorgiou and Smyrlis \cite[Proposition 5 ]{14}, we know that problem
\eqref{eq4} has a unique positive solution $\tilde{u}\in \operatorname{int}C_+$.

Let $\delta>0$ be as postulated by hypothesis (H1)(iii) and let
$$
0<t\leq \min\big\{1,\frac{\delta}{\|u\|_{\infty}}\big\}.
$$
We set $\underline{u}=t\tilde{u}$. Then $\underline{u}\in \operatorname{int}C_+$
and we have
\begin{equation} \label{eq5}
\begin{aligned}
 -\Delta_p\underline{u}(z)=t^{p-1}[-\Delta_p\tilde{u}(z)]
&= t^{p-1}\tilde{u}(z)^{-\mu} \\
&\leq \underline{u}(z)^{-\mu}\quad (\text{since } 0<t\leq 1) \\
&\leq \underline{u}(z)^{-\mu}+f(z,\underline{u}(z)) \text{ for almost all }
 z\in\Omega
\end{aligned}
\end{equation}
(see \cite{14}, note that $\underline{u}(z)\in\left(0,\delta\right]$ for all
$z\in\overline{\Omega}$ and see hypothesis (H1)(iii)).
Also note that
$\underline{u}\leq w$.

We introduce the following truncation of the reaction term in \eqref{eq1}:
\begin{equation}\label{eq6}
 \hat{f}(z,x)=\begin{cases}
  \underline{u}(z)^{-\mu}+f(z,\underline{u}(z))&\text{if } x<\underline{u}(z)\\
  x^{-\mu}+f(z,x)&\text{if } \underline{u}(z)\leq x\leq w(z)\\
  w(z)^{-\mu}+f(z,w(z))&\text{if } w(z)<x.
 \end{cases}
\end{equation}
This is a Carath\'eodory function. We set $\hat{F}(z,x)=\int^x_0\hat{f}(z,s)ds$
and consider the functional $\hat{\varphi}:W^{1,p}_0(\Omega)\to{\mathbb R}$
 defined by
$$
\hat{\varphi}(u)=\frac{1}{p}\|Du\|^p_p-\int_{\Omega}\hat{F}(z,u)dz\quad
 \text{for all } u\in W^{1,p}_0(\Omega).
$$

By Papageorgiou and Smyrlis \cite[Proposition 3]{14} we have
$\hat{\varphi}\in C^1(W^{1,p}_0({\mathbb R}))$.

In what follows, we denote by $[\underline{u},w]$ the order interval
$$
[\underline{u},w]=\{u\in W^{1,p}_0(\Omega):\underline{u}(z)
\leq u(z)\leq w(z) \text{ for almost all } z\in\Omega\}.
$$
Also, we denote by $\operatorname{int}_{C^1_0(\overline{\Omega})}
[\underline{u},w]$ the interior in the $C^1_0(\overline{\Omega})$-norm
topology of $[\underline{u},w]\cap C^1_0(\overline{\Omega})$.

In the next proposition we produce a positive smooth solution located
in the above order interval.

\begin{proposition}\label{prop4}
 If hypotheses {\rm (H1)} hold, then problem \eqref{eq1} has a positive
solution $u_0\in \operatorname{int}_{C^1_0(\overline{\Omega})}[\underline{u},w]$.
\end{proposition}

\begin{proof}
 We know that $\underline{u}\in \operatorname{int}C_+$. So, using
Marano and Papageorgiou \cite[Proposition 2.1]{10} we can find $c_0>0$ such that
\[
\hat{u}^{1/p'}_{1}\leq c_0\underline{u}\quad
 \Rightarrow\quad \underline{u}^{-\mu}\leq c_0^{\mu}\hat{u}_1^{-\mu/p'}.
\]
Hence using the lemma of Lazer and McKenna \cite{9}, we have that
 $$
\underline{u}^{-\mu}\in L^{p'}(\Omega).
$$

 Therefore by \eqref{eq5} we see that $\hat{\varphi}(\cdot)$ is coercive.
Also, using the Sobolev embedding theorem, we see that $\hat{\varphi}$
is sequentially weakly lower semicontinuous. So, by the Weierstrass-Tonelli
theorem, we can find $u_0\in W^{1,p}_0(\Omega)$ such that
 \begin{equation} \label{eq7}
\begin{aligned}
  &\hat{\varphi}(u_0)=\inf[\hat{\varphi}(u):u\in W^{1,p}_0(\Omega)], \\
  &\Rightarrow \hat{\varphi}'(u_0)=0, \\
  &\Rightarrow \langle A(u_0),h\rangle=\int_{\Omega}\hat{f}(z,u_0)hdz
 \text{ for all } h\in W^{1,p}_0(\Omega).
 \end{aligned}
\end{equation}
In \eqref{eq7} we first choose $h=(\underline{u}-u_0)^+\in W^{1,p}_0(\Omega)$.
Then
 \begin{align*}
  \langle A(u_0),(\underline{u}-u_0)^+\rangle
&=\int_{\Omega}[\underline{u}^{-\mu}+f(z,\underline{u})](\underline{u}-u_0)^+dz
\quad \text{(see \eqref{eq6})} \\
&\geq \langle A(\underline{u}),(\underline{u}-u_0)^+\rangle\quad
 (\text{see \eqref{eq5}})
\end{align*}
which implies
\[
 \langle A(\underline{u})-A(u_0),(\underline{u}-u_0)^+\rangle\leq 0,
\]
and this implies $\underline{u}\leq u_0$.

Next, in \eqref{eq7} we choose $h=(u_0-w)^+\in W^{1,p}_0(\Omega)$
(see hypothesis (H1)(i)). Then
  \begin{align*}
  \langle A(u_0),(u_0-w)^+\rangle
&=\int_{\Omega}[w^{-\mu}+f(z,w)](u_0-w)^+dz\\
&\leq\langle A(w),(u_0-w)^+\rangle\quad (\text{see hypothesis (H1)(i)}),
\end{align*}
which implies
\[
\langle A(u_0)-A(w),(u_0-w)^+\rangle\leq 0,
\]
and this implies $u_0\leq w$.
So, we have proved that
 \begin{equation}\label{eq8}
  u_0\in[\underline{u},w]=\{u\in W^{1,p}_0(\Omega):\underline{u}(z)
\leq u_0(z)\leq w(z)\quad \text{ for almost all } z\in\Omega\}.
 \end{equation}

 Clearly, $u_0\neq \underline{u}$ (see hypothesis (H1)(iii)) and $u_0\neq w$
(see hypothesis (H1)(i)).
From \eqref{eq6}, \eqref{eq7}, \eqref{eq8}, we have
\[
\langle A(u_0),h\rangle=\int_{\Omega}[u_0^{-\mu}+f(z,u_0)]hdz,\quad
 0\leq u_0^{-\mu}\leq \underline{u}^{-\mu}\in L^p(\Omega)
\]
which implies
 \begin{equation}\label{eq9}
 -\Delta_pu_0(z)=u_0(z)^{-\mu}+f(z,u_0(z)) \quad\text{for a.a. }
 z\in\Omega,\; u_0|_{\partial\Omega}=0,
\end{equation}
see \cite{14}.

 Also, by  Gilbarg and Trudinger \cite[ Lemma 14.16 p. 355]{6} we know
that there exists small $\delta_0>0$ such that, if
$\Omega_{\delta_0}=\{z\in\Omega:d(z,\partial\Omega)<\delta_0\}$, then
 $$
d\in \operatorname{int}C_+(\overline{\Omega}_{\delta_0}),
$$
 where $d(\cdot)=d(\cdot,\partial\Omega)$.
Let $D^*=\overline{\Omega}\backslash \Omega_{\delta_0}$.
Setting $C(D^*)_+=\{h\in C(D^*):h(z)\geq 0 \text{ for all } z\in D^*\}$,
 we have $d\in \operatorname{int}C(D^*)_+\subseteq \operatorname{int}C_+(D^*)$.
Then as before, via Marano and Papageorgiou \cite[Proposition 2.1 ]{10}
we  find $0<c_1<c_2$ such that
 \begin{equation}\label{eq10}
  c_1d\leq \underline{u}\leq c_2d.
 \end{equation}
Then by \eqref{eq9}, \eqref{eq10}, hypotheses (H1)(i), (H1)(iv) and
 Giacomoni and Saudi \cite[Theorem B.1]{4}, we have
 $$
u_0\in \operatorname{int}C_+.
$$

 Now let $\rho=\|w\|_{\infty}$ and let $\hat{\xi}_{\rho}>0$
be as postulated by hypothesis (H1)(iv). We have
 \begin{align*}
  &-\Delta_p u_0(z)-u_0(z)^{-\mu}+\hat{\xi}_{\rho}u_0(z)^{p-1}\\
  &= f(z,u_0(z))+\hat{\xi}_{\rho}u_0(z)^{p-1}\quad (\text{see \eqref{eq9}})\\
  &\geq f(z,\underline{u}(z))+\hat{\xi}_{\rho}\underline{u}(z)^{p-1}\quad (\text{see \eqref{eq8} and hypothesis}\ (H1)(iv))\\
  &> \hat{\xi}_{\rho}\underline{u}(z)^{p-1}\quad (\text{see hypothesis(H1)(ii)})\\
  &\geq -\Delta_p\underline{u}(z)-\underline{u}(z)^{-\mu}+\hat{\xi}_{\rho}
 \underline{u}(z)^{p-1}\quad (\text{see \eqref{eq5}})\ \text{for almost all }
 z\in\Omega.
 \end{align*}

 Hence, invoking Proposition \ref{prop4} of Papageorgiou and Smyrlis \cite{14},
we have
 $$
u_0-\underline{u}\in \operatorname{int}C_+.
$$
From the hypothesis on the function $w(\cdot)$ (see (H1)(i)), we see that
 $$
D_0=\{z\in\Omega:u_0(z)=w(z)\}\ \text{is compact in}\ \Omega.
$$
Then we can find an open set $\mathcal{U}\subseteq\Omega$ with Lipschitz boundary,
 such that
 $$
D_0\subseteq\mathcal{U}\subseteq\overline{\mathcal{U}}\subseteq\Omega
 \text{ and } d(z,D_0)\leq \delta_1 \text{ for all }
 z\in\overline{\mathcal{U}}, \text{ with } \delta_1>0.
$$
 Let $\epsilon>0$ be such that
 \begin{equation}\label{eq11}
  u_0(z)+\epsilon\leq w(z) \text{ for all } z\in\partial\mathcal{U}
 \end{equation}
 (such an $\epsilon>0$ exists since $\partial\Omega$ is compact and
$w-u_0\in C(\overline{\Omega})$).

 Exploiting the uniform continuity of the map $x\mapsto x^{p-1}$ on $[0,\rho]$
we can find $\delta_2>0$ such that
 \begin{equation}\label{eq12}
  \hat{\xi}_{\rho}|x^{p-1}-v^{p-1}|\leq\epsilon\quad \text{for all }
 x,v\in[\min_{\overline{\mathcal{U}}}u_0,\max_{\overline{\mathcal{U}}}w],\;
 |x-v|\leq\delta_2.
 \end{equation}

 Similarly, the uniform continuity of $x\mapsto x^{-\mu}$ on any compact subset
of $(0,+\infty)$, implies that we can find $\delta_3\in\left(0,\delta_2\right]$
such that
 \begin{equation}\label{eq13}
  |x^{-\mu}-v^{-\mu}|\leq\epsilon\quad \text{for all }
 x,v\in\big[\frac{\hat{c}}{2},\|w\|_{\infty}\big],\; |x-v|\leq\delta_2.
 \end{equation}
Then choosing $\delta_1\in(0,\delta_3)$ small enough and
$\tilde{\delta}\in(0,\delta_1)$ we have
 \begin{equation} \label{eq14}
\begin{aligned}
  &-\Delta_p(u_0+\tilde{\delta})(z)+\hat{\xi}_{\rho}(u_0+\tilde{\delta})(z)^{p-1} \\
  &\leq -\Delta_pu_0(z)+\tilde{\xi}_{\rho}u_0(z)^{p-1}+\epsilon
 \quad (\text{see \eqref{eq12}}) \\
  &= u_0(z)^{-\mu}+f(z,u_0(z))+\hat{\xi}_{\rho}u_0(z)^{p-1}+\epsilon
\quad (\text{see \eqref{eq9}}) \\
  &\leq w(z)^{-\mu}+f(z,w(z))+\hat{\xi}_{\rho}w(z)^{p-1}+2\epsilon
\quad (\text{see \eqref{eq13}, \eqref{eq8},  (H1)(iv)}) \\
  &\leq -c_{\overline{\mathcal{U}}}+2\epsilon+\hat{\xi}_{\rho}w(z)^{p-1}
 \text{ for almost all } z\in\Omega \quad (\text{see  (H1)(i)}).
 \end{aligned}
\end{equation}

 Choosing $\epsilon\in\left(0,c_{\overline{\mathcal{U}}}/2\right)$ and using
 once more hypothesis (H1)(i), we deduce from \eqref{eq14} that
 \begin{equation}\label{eq15}
-\Delta_p(u_0+\tilde{\delta})+\hat{\xi}_{\rho}(u_0+\tilde{\delta})^{p-1}
\leq-\Delta_pw+\hat{\xi}_{\rho}w^{p-1}\ \text{in }
 W^{1,p}_0(\Omega)^*=W^{-1,p'}(\Omega).
 \end{equation}
From \eqref{eq15}, \eqref{eq11} and the weak comparison principle of
Tolksdorf \cite[Lemma 3.1]{17}, we have
 $$
(u_0+\tilde{\delta})(z)\leq w(z) \text{ for all } z\in\overline{\mathcal{U}}.
$$
But $D_0\subseteq\overline{\mathcal{U}}$. Therefore $D_0=\emptyset$ and so
 $$
0<(w-u_0)(z)\text{ for all } z\in\overline{\Omega}.
$$
 We conclude that
 $$
u_0\in \operatorname{int}_{C^1_0(\overline{\Omega})}[\underline{u},w].
$$
The proof is now complete.
\end{proof}

Next we produce a second positive smooth solution for problem \eqref{eq1}.

\begin{proposition}\label{prop5}
 If hypotheses {\rm (H1)} hold, then  \eqref{eq1} has a second positive solution
$\hat{u}\in \operatorname{int} C_+$.
\end{proposition}

\begin{proof}
 Consider the following truncation of the reaction term in \eqref{eq1}:
 \begin{equation} \label{eq16}
  g(z,x)=\begin{cases}
  \underline{u}(z)^{-\mu}+f(z,\underline{u}(z))&\text{if } u\leq\underline{u}(z)\\
  x^{-\mu}+f(z,x)&\text{if } \underline{u}(z)<x.
  \end{cases}
 \end{equation}

 This is a Carath\'eodory function. We set $G(z,x)=\int^x_0g(z,s)ds$ and
consider the functional $\varphi_0:W^{1,p}_0(\Omega)\to{\mathbb R}$ defined by
 $$
\varphi_0(u)=\frac{1}{p}\|Du\|^p_p-\int_{\Omega}G(z,u)dz\quad \text{for all }
 u\in W^{1,p}_0(\Omega).
$$
As before,  Papageorgiou and Smyrlis \cite[Proposition 3]{14} implies that
 $$
\varphi_0\in C^1(W^{1,p}_0(\Omega)).
$$
\smallskip

\noindent\textbf{Claim.}  $\varphi_0$ satisfies the C-condition.

We consider a sequence $\{u_n\}_{n\geq 1}\subseteq W^{1,p}_0(\Omega)$ such that
 \begin{gather}
|\varphi_0(u_n)|\leq M_1\quad \text{for some $M_1>0$ and for all }
 n\in {\mathbb N},\label{eq17}\\
 (1+\|u_n\|)\varphi'_0(u_n)\to 0\quad \text{in } W^{-1,p'}(\Omega) \text{ as }
 n\to\infty.\label{eq18}
 \end{gather}
From \eqref{eq17} we have
 \begin{equation} \label{eq19}
 \big|\langle A(u_n),h\rangle-\int_{\Omega}g(z,u_n)hdz\big|
\leq\frac{\epsilon_n\|h\|}{1+\|u_n\|}
\end{equation}
for all $h\in W^{1,p}_0(\Omega)$ with $\epsilon_n\to 0^+$.

In \eqref{eq19} we choose $h=-u^-_n\in W^{1,p}_0(\Omega)$. Then
\[
\|Du^-_n\|^p_p-\int_{\Omega}[\underline{u}^{-\mu}+f(z,\underline{u})](-u^-_n)dz
\leq\epsilon_n\quad \text{for all $n\in {\mathbb N}$, (see \eqref{eq16})}
\]
which implies
\begin{equation} \label{eq20}
\begin{aligned}
&\|u^-_n\|^p\leq c_3\|u^-_n\|\quad \text{for some $c_3>0$ and for all }
 n\in {\mathbb N}, \\
&\Rightarrow \{u^-_n\}_{n\geq 1}\subseteq W^{1,p}_0(\Omega) \text{ is bounded}.
 \end{aligned}
\end{equation}
 Suppose that $\{u^+_n\}_{n\geq 1}\subseteq W^{1,p}_0(\Omega)$ is unbounded.
 By passing to a subsequence if necessary, we may assume that
 \begin{equation}\label{eq21}
  \|u^+_n\|\to\infty
 \end{equation}
Let $y_n=\frac{u^+_n}{\|u^+_n\|}$, $n\in {\mathbb N}$.
Then $\|y_n\|=1$, $y_n\geq 0$ for all $n\in {\mathbb N}$. So, we may assume that
 \begin{equation}\label{eq22}
  y_n\stackrel{w}{\to}y\ \text{in}\ W^{1,p}_0(\Omega)\quad \text{and}\quad
 y_n\to y \text{ in } L^p(\Omega),\; y\geq 0.
 \end{equation}
From \eqref{eq19} and \eqref{eq20} we have
 \[
\big|\langle A(u^+_n),h\rangle-\int_{\Omega}g(z,u^+_n)hdz\big|
\leq c_4\|h\|\quad \text{for some $c_4>0$ and  all } n\in {\mathbb N}
\]
which implies
 \begin{equation} \label{eq23}
\big|\langle A(y_n),h\rangle-\int_{\Omega}\frac{N_g(u^+_n)}{\|u^+_n\|^{p-1}}hdz
\big|\leq\frac{c_4\|h\|}{\|u^+_n\|^{p-1}}\quad \text{for all } n\in {\mathbb N}.
\end{equation}

Hypotheses (H1)(i) and (H1)(i)(ii) imply that there exists $c_5>0$ such that
 $$
|f(z,x)|\leq c_5(1+x^{p-1})\quad \text{for almost all $z\in\Omega$ and all } x>0.
$$
From this growth estimate and \eqref{eq16}, it follows that
 $$
\Big\{\frac{N_g(u^+_n)}{\|u^+_n\|^{p-1}}\Big\}_{n\geq 1}
\subseteq L^{p'}(\Omega) \text{ is bounded}.
$$
So, by passing to a suitable sequence if necessary and using hypothesis
(H1)(ii) we have
 \begin{equation}\label{eq24}
\frac{N_g(u^+_n)}{\|u^+_n\|^{p-1}}\stackrel{w}{\to}\tilde{\eta}(z)y^{p-1}\quad
 \text{in $L^{p'}(\Omega)$ as } n\to\infty,
 \end{equation}
with $\hat{\lambda}_1\leq\tilde{\eta}(z)\leq\eta(z)$ for almost all $z\in\Omega$,
see Aizicovici, Papageorgiou and Staicu \cite[proof of Proposition 16)]{1}.

 Recall that $\underline{u}^{-\mu}\in L^{p'}(\Omega)$. Therefore
\[
\big|\int_{\Omega}\underline{u}^{-\mu}hdz\big|
\leq c_6\|h\|\quad \text{for some $c_6>0$ and all } h\in W^{1,p}_0(\Omega)
\]
which implies
\begin{equation}\label{eq25}
\frac{1}{\|u^+_n\|^{p-1}}\int_{\Omega}\underline{u}^{-\mu}hdz\to 0\quad
 \text{as $n\to\infty$,  (see \eqref{eq21})}.
 \end{equation}
If in \eqref{eq23} we choose $h=y_n-y\in W^{1,p}_0(\Omega)$ and pass to
the limit as $n\to\infty$, then using \eqref{eq22}, \eqref{eq24}, \eqref{eq25}
 we have
 $\lim_{n\to\infty}\langle A(y_n),y_n-y\rangle=0$ which implies
\begin{equation}\label{eq26}
y_n\to y \text{ in } W^{1,p}_0(\Omega),\quad \|y\|=1,\; y\geq 0\;
(\text{see Proposition \ref{prop2}}).
 \end{equation}
So, if in \eqref{eq23} we pass to the limit as $n\to\infty$ and use \eqref{eq24},
\eqref{eq25}, \eqref{eq26} to obtain
\[
\langle A(y),h\rangle=\int_{\Omega}\tilde{\eta}(z)y^{p-1}hdz\quad
 \text{for all } h\in W^{1,p}_0(\Omega)
\]
which implies
 \begin{equation}\label{eq27}
-\Delta_py(z)=\tilde{\eta}(z)y(z)^{p-1}\quad \text{for almost all }
z\in\Omega,\quad y|_{\partial \Omega}=0.
 \end{equation}
Recall that
 $$
\hat{\lambda}_1\leq\tilde{\eta}(z)\leq\eta(z)\quad \text{for almost all }
 z\in\Omega\ (\text{see \eqref{eq24}}).
$$
We first assume that $\hat{\lambda}_1\not\equiv\tilde{\eta}$.
Then using Proposition \ref{prop3} we have
 $$
\hat{\lambda}_1(\tilde{\eta})<\hat{\lambda}_1(\hat{\lambda_1})=1.
$$
Also, from \eqref{eq27} and since $\|y\|=1$ (hence $y\neq 0$, see \eqref{eq26}),
we infer that $y(\cdot)$ must be nodal, a contradiction to \eqref{eq22}.

 Next, we assume that $\tilde{\eta}(z)=\hat{\lambda}_1$ for almost all
$z\in\Omega$. It follows from \eqref{eq27} that
\[
y=\vartheta\hat{u}_1\quad \text{with $\vartheta>0$, see \eqref{eq26}}.
\]
Then $y\in \operatorname{int}C_+$ and so $y(z)>0$ for all $z\in\Omega$. Therefore
 \begin{equation}
u^+_n(z)\to+\infty\ \text{for all}\ z\in\Omega\quad \text{as }
 n\to\infty,\label{eq28}
\end{equation}
which implies
\[
f(z,u^+_n(z))u^+_n(z)-pF(z,u^+_n(z))\to-\infty
\]
for almost all $z\in\Omega$ as $n\to\infty$, see hypothesis (H1)(ii).
 This in turn implies
\begin{equation}
\int_{\Omega}[f(z,u^+_n)u^+_n-pF(z,u^+_n)]dz\to-\infty\quad (\text{by Fatou's lemma}).
\label{eq29}
\end{equation}
From \eqref{eq19} with $h=u^+_n\in W^{1,p}_0(\Omega)$, we have
 \begin{equation}\label{eq30}
  -\|Du^+_n\|^p_p+\int_{\Omega}g(z,u^+_n)u^+_ndz\geq-\epsilon_n\quad
 \text{for all } n\in {\mathbb N}.
 \end{equation}
On the other hand, from \eqref{eq17} and \eqref{eq20}, we have
 \begin{equation}\label{eq31}
  \|Du^+_n\|^p_p-\int_{\Omega}pG(z,u^+_n)dz\geq-M_2\quad
 \text{for some $M_2>0$ and all } n\in {\mathbb N}.
 \end{equation}
Adding \eqref{eq30} and \eqref{eq31}, we obtain
\[
\int_{\Omega}[g(z,u^+_n)u^+_n-pG(z,u^+_n)]dz\geq-M_3\quad
\text{for some $M_3>0$ and all } n\in {\mathbb N}
\]
which implies
 \begin{equation}\label{eq32}
\int_{\Omega}[f(z,u^+_n)u^+_n-pF(z,u^+_n)]dz\geq-M_4
 \end{equation}
for some $M_4>0$ and all $n\in{\mathbb N}$ (see \eqref{eq16} and \eqref{eq28}).

 Comparing \eqref{eq29} and \eqref{eq32}, we have a contradiction.
This proves that
 \begin{align*}
  &\{u^+_n\}_{n\geq 1}\subseteq W^{1,p}_0(\Omega) \text{ is bounded},\\
  &\Rightarrow \{u_n\}_{n\geq 1}\subseteq W^{1,p}_0(\Omega) \text{ is bounded
(see \eqref{eq20})}.
 \end{align*}
So, we assume that
 $$
u_n\stackrel{w}{\to}u \text{ in } W^{1,p}_0(\Omega)\quad \text{and}\quad
 u_n\to u \text{ in } L^p(\Omega).
$$
Then we obtain
 \begin{equation}\label{eq33}
  \int_{\Omega}g(z,u_n)(u_n-u)dz\to 0\quad \text{as } n\to\infty.
 \end{equation}
If in \eqref{eq19} we choose $h=u_n-u\in W^{1,p}_0(\Omega)$, then
 \begin{align*}
  &\lim_{n\to\infty}\langle A(u_n),u_n-u\rangle=0,\\
  &\Rightarrow u_n\to u\quad \text{in } W^{1,p}_0(\Omega)\
 (\text{see Proposition \ref{prop2}}).
 \end{align*}
This proves the claim.

 Note that
 \begin{equation}\label{eq34}
  \hat{\varphi}\big|_{[\underline{u},w]}
=\varphi_0\big|_{[\underline{u},w]}\quad (\text{see \eqref{eq6} and \eqref{eq16}}).
 \end{equation}
From the proof of Proposition \ref{prop4} we know that
$u_0\in \operatorname{int}_{C^1_0(\overline{\Omega})}[\underline{u},w]$
is a minimizer of $\hat{\varphi}$. Hence it follows from \eqref{eq34}
that $u_0$ is a local $C^1_0(\overline{\Omega})$-minimizer of $\varphi_0$.
Invoking  Giacomoni and Saudi \cite[Theorem 1.1]{4}, we can say that $u_0$
is a local $W^{1,p}_0(\Omega)$-minimizer of $\varphi_0$. Using \eqref{eq16}
we can easily see that
 \begin{align*}
  K_{\varphi_0}
&= \{u\in W^{1,p}_0(\Omega):\varphi'_0(u)=0\}
 \subseteq[\underline{u})\cap C_+\\
&= \{u\in C^1_0(\overline{\Omega}):\underline{u}(z)\leq u(z) \text{ for all }
 z\in\overline{\Omega}\}.
\end{align*}
So, we may assume that $K_{\varphi_0}$ is finite or otherwise we already
have an infinity of positive smooth solutions of \eqref{eq1}. Since $u_0$
is a local minimizer of $\varphi_0$ we can find $\rho\in(0,1)$ small such that
 \begin{equation}\label{eq35}
  \varphi_0(u_0)<\inf[\varphi_0(u):\|u-u_0\|=\rho]=m_{\rho}
 \end{equation}
 (see Aizicovici, Papageorgiou and Staicu \cite[proof of Proposition 29]{1}).

 Hypothesis (H1)(ii) implies that given any $\xi>0$, we can find $M_5=M_5(\xi)>0$
such that
 \begin{equation}\label{eq36}
  f(z,x)x-pF(z,x)\leq-\xi \text{ for almost all $z\in\Omega$ and all } x\geq M_5.
 \end{equation}
We have
\begin{align*}
\frac{d}{dx}\Big(\frac{F(z,x)}{x^p}\Big)\
&=\frac{f(z,x)x^{2p}-px^{p-1}F(z,x)}{x^{2p}} \\
&=\frac{f(z,x)x-pF(z,x)}{x^{p+1}} \\
&\leq-\frac{\xi}{x^{p+1}}
\end{align*}
for almost all $z\in\Omega$ and all
$x\geq M_5$, see \eqref{eq36}.
This implies
 \begin{equation}\label{eq37}
\frac{F(z,x)}{x^p}-\frac{F(z,y)}{y^p}
\leq\frac{\xi}{p}\big[\frac{1}{x^p}-\frac{1}{y^p}\big],
 \end{equation}
for almost all $z\in\Omega$, for all $x\geq y\geq M_5$.

 Hypothesis (H1)(iii) implies
 \begin{equation}\label{eq38}
\hat{\lambda}_1\leq\liminf_{x\to+\infty}\frac{pF(z,x)}{x^p}
\leq\limsup_{x\to+\infty}\frac{pF(z,x)}{x^p}\leq\eta(z)
 \end{equation}
uniformly for almost all $z\in\Omega$.

In \eqref{eq37} we pass to the limit as $x\to+\infty$ and use \eqref{eq38}.
 We obtain that
$\hat{\lambda}_1y^p-pF(z,y)\leq-\xi$ for almost all $z\in\Omega$ and all
$y\geq M_5$. This implies
 \begin{equation}\label{eq39}
\hat{\lambda}_1y^p-pF(z,y)\to-\infty\quad \text{as } y\to+\infty
\text{ uniformly for a.a .} z\in\Omega.
 \end{equation}
For $t>0$ big (so that $t\hat{u}_1\geq \underline{u}$, recall that
$\hat{u}_1\in \operatorname{int}C_+$), we have
\[
\varphi_0(t\hat{u}_1)\leq\frac{t^p}{p}\hat{\lambda}_1-\int_{\Omega}F(z,t\hat{u}_1)dz
+c_7\quad \text{for some $c_7>0$, see \eqref{eq16}}
\]
which implies
\[
p\varphi_0(t\hat{u}_1)\leq\int_{\Omega}[\hat{\lambda}_1(t\hat{u}_1)^p
-pF(z,t\hat{u}_1)]dz+pc_7,
\]
which in turn implies
 \begin{equation}\label{eq40}
p\varphi_0(t\hat{u}_1)\to-\infty\quad (\text{see \eqref{eq39} and use Fatou's lemma}).
 \end{equation}
Then \eqref{eq35}, \eqref{eq40} and the claim permit the use of Theorem \ref{thm1}
(the mountain pass theorem) and so we can find $\hat{u}\in W^{1,p}_0(\Omega)$
such that
 \begin{equation}\label{eq41}
  \hat{u}\in K_{\varphi_0}\quad \text{and}\quad m_{\rho}\leq\varphi_0(\hat{u}).
 \end{equation}
It follows from \eqref{eq35} and \eqref{eq41} that
$\hat{u}\neq u_0$, $\hat{u}\in[\underline{u})\cap C_+$ and so
$\hat{u}\in \operatorname{int}C_+$ is the second positive smooth solution
of problem \eqref{eq1}.
\end{proof}

So, we can state the following multiplicity theorem for problem \eqref{eq1}

\begin{theorem}\label{thm6}
 If hypotheses {\rm (H1)} hold, then problem \eqref{eq1} has at least two
 positive smooth solutions
$u_0$ and $\hat{u}$ in $\operatorname{int}C_+$.
\end{theorem}

\subsection*{Acknowledgments}
This research was supported by the Slovenian Research Agen\-cy grants P1-0292,
J1-8131, and J1-7025.
 V. D. R\u adulescu acknowledges the support through a grant of the Ministry
of Research and Innovation, CNCS--UEFISCDI, project number
PN-III-P4-ID-PCE-2016-0130, within PNCDI III.

\newpage

\end{document}